\documentclass[11pt]{amsart}
\usepackage{booktabs}
\usepackage{a4wide}
\usepackage{amsmath, amsthm, amssymb, amsfonts} % standard math libraries
\usepackage{mathtools} % extension package to amsmath
\usepackage[margin=1in, footskip=0.25in]{geometry}
\usepackage[dvipsnames]{xcolor} % more colors

\usepackage{graphicx} % inserting images

\usepackage{hyperref} % hyperlinks

\usepackage{enumitem} % environments for enumeration
\usepackage{calc} % calculation with length units

\usepackage{tikz}
\usetikzlibrary{
    arrows,
    calc,
    decorations.pathreplacing,
    fit,
    positioning,
    shapes.arrows,
    shapes.callouts,
    shapes.geometric,
}

\usepackage{float} % here for H placement parameter

\usepackage[
    backend=biber,
     style=numeric,
     citestyle=numeric,
     doi=false,
     isbn=false,
     url=false,
     maxnames=50,
     giveninits=true]{biblatex}

\usepackage{subfiles}

\newtheorem{theorem}{Theorem}[section]
\newtheorem{lemma}[theorem]{Lemma}

\newtheorem{claim}{Claim}

\newtheorem{proposition}[theorem]{Proposition}

\newtheorem{remark}[theorem]{Remark}

\makeatletter
\def\pgfsetinvisiblelayers#1{\def\pgf@layers@invisible{#1}}
\def\pgfshowalllayers{\let\pgf@layers@invisible=\pgfutil@empty}
\pgfshowalllayers

\def\pgf@dolayer#1,#2,\relax{%
  \edef\pgf@marshal{\noexpand\pgfutil@in@{,#1,}{,\pgf@layers@invisible,}}%
  \pgf@marshal%
  \ifpgfutil@in@% Yep. So do nothing.
  \else% Nope. So, insert box.
    \def\pgf@test{#1}%
    \ifx\pgf@test\pgf@maintext%
      \box\pgf@layerbox@main%
    \else%
      \pgfsys@beginscope%
        \expandafter\box\csname pgf@layerbox@#1\endcsname%
      \pgfsys@endscope%bi
    \fi%
  \fi
  \def\pgf@test{#2}%
  \ifx\pgf@test\pgfutil@empty%
  \else%
    \pgf@dolayer#2,\relax%
  \fi%
}
\makeatother

\pgfdeclarelayer{bg}    
\pgfdeclarelayer{mg}    
\pgfdeclarelayer{inv}    
\pgfsetinvisiblelayers{inv}
\pgfsetlayers{inv,bg,mg,main} 

\tikzset{fit margins/.style={/tikz/afit/.cd,#1,
    /tikz/.cd,
    inner xsep=\pgfkeysvalueof{/tikz/afit/left}+\pgfkeysvalueof{/tikz/afit/right},
    inner ysep=\pgfkeysvalueof{/tikz/afit/top}+\pgfkeysvalueof{/tikz/afit/bottom},
    xshift=-\pgfkeysvalueof{/tikz/afit/left}+\pgfkeysvalueof{/tikz/afit/right},
    yshift=-\pgfkeysvalueof{/tikz/afit/bottom}+\pgfkeysvalueof{/tikz/afit/top}},
    afit/.cd,left/.initial=2pt,right/.initial=2pt,bottom/.initial=2pt,top/.initial=2pt}

\makeatletter
\NewDocumentCommand {\getnodedimen} {O{\nodewidth} O{\nodeheight} m} {
	\begin{pgfinterruptboundingbox}
	\begin{scope}[local bounding box=bb@temp]
		\node[inner sep=0pt, fit=(#3)] {};
	\end{scope}
	\path ($(bb@temp.north east)-(bb@temp.south west)$);
	\end{pgfinterruptboundingbox}
	\pgfgetlastxy{#1}{#2}
}
\makeatother

\makeatletter
\DeclareRobustCommand{\rvdots}{%
  \vbox{
    \baselineskip4\p@\lineskiplimit\z@
    \hbox{.}\hbox{.}\hbox{.}
  }}
\makeatother

\makeatletter
\gdef\gsetlength#1#2{%
  \begingroup
    \setlength\skip@{#2}% Local assignment to a scratch register.
    \global#1=\skip@    % Global assignement to #1;
  \endgroup             % \skip@ is restored by end of group.
}
\makeatother

\makeatletter
\newcommand{\deflen}[2]{%
  \@ifundefined{#1}{\prefix@deflen{#1}{#2}}{\prefix@setlen{#1}{#2}}
}
\newcommand{\prefix@deflen}[2]{%
  \expandafter\newlength\csname #1\endcsname
  \expandafter\setlength\csname #1\endcsname{#2}%
}
\newcommand{\prefix@setlen}[2]{%
  \expandafter\setlength\csname #1\endcsname{#2}%
}
\makeatother

\makeatletter
\newcommand{\gdeflen}[2]{%
  \@ifundefined{#1}{\prefix@gdeflen{#1}{#2}}{\prefix@gsetlen{#1}{#2}}
}
\newcommand{\prefix@gdeflen}[2]{%
  \expandafter\newlength\csname #1\endcsname
  \expandafter\gsetlength\csname #1\endcsname{#2}%
}
\newcommand{\prefix@gsetlen}[2]{%
  \expandafter\gsetlength\csname #1\endcsname{#2}%
}
\makeatother
 \deflen{itemSize}{25pt}
\deflen{itemshift}{10pt}
\deflen{partshift}{6em}
\deflen{horizontalmargin}{6pt}
\deflen{verticalmargin}{3.5pt}
\deflen{extH}{6pt}
\deflen{extV}{4pt}
\deflen{thinLine}{0.5pt}
\deflen{partLabelShift}{8pt}

\tikzset{%
    block/.style={
        rectangle,
        draw=black,
        thick,
        align=center,
        fit margins={left=\horizontalmargin,right=\horizontalmargin,bottom=\verticalmargin,top=\verticalmargin},
    },
    item/.style={
        circle,
        thick,
        align=center,
        minimum size=\itemSize,
        align=center,
    }
}

\pgfkeys{
    /cbNodes/.is family, /cbNodes,
    nrFirst/.initial = 4,
    nrLast/.initial = 1,
}
\newcommand*{\drawNodes}[1][]{%
    \begingroup
    \pgfkeys{/cbNodes, #1}%
    \newcommand*{\Nodes}[1]{\pgfkeysvalueof{/cbNodes/##1}}
    
    \foreach \j in {1,2}{
        \ifthenelse{\j=1}{%
            \def\s{v}
            \deflen{y}{0pt}
            \def\partLabel{delta}
            \def\partName{\Delta}
        }{%
            \def\s{u}
            \deflen{y}{-\partshift}
            \def\partLabel{deltaPrime}
            \def\partName{\Delta'}
        }
        \deflen{x}{0pt}
        \foreach \i in {1,...,\Nodes{nrFirst}}{
            \node [anchor=west, item] at (\x, \y) (\s_\i) {${\s}_{\i}$};
            \getnodedimen{\s_\i}
            \gsetlength{\x}{\dimexpr \x + \nodewidth + \itemshift \relax};
        }
        \node [anchor=west, item] at (\x, \y) (\s_dots) {$\cdots$};
        \getnodedimen{\s_dots}
        \gsetlength{\x}{\dimexpr \x + \nodewidth + \itemshift \relax};
        \foreach \i in {1,...,\Nodes{nrLast}}{
            \pgfmathtruncatemacro\k{\Nodes{nrLast} - \i}
            \ifthenelse{\k=0}{%
                \def\suf{n}
                \def\ind{n}
            }{%
                \def\suf{nMinus{\k}}
                \def\ind{n - \k}
            }
            \node [anchor=west, item] at (\x, \y) (\s_\suf) {${\s}_{\ind}$};
            \getnodedimen{\s_\suf}
            \gsetlength{\x}{\dimexpr \x + \nodewidth + \itemshift \relax};
        }
        \node (\partLabel) [block, fit = (\s_1) (\s_1) (\s_1) (\s_n)] {};
        \node [right=\partLabelShift of \partLabel] {$\partName$};
    }

    \endgroup
}

\DeclareMathOperator{\Sym}{Sym}
\DeclareMathOperator{\Alt}{Alt}
\DeclareMathOperator{\Aut}{Aut}

\DeclareMathOperator{\Diag}{Diag}

\DeclareMathOperator{\PSL}{PSL}
\DeclareMathOperator{\PGL}{PGL}
\DeclareMathOperator{\soc}{soc}

\newcommand{\N}{\mathbb{N}}

\begin{document}

\title{The distinguishing number of  complete bipartite and crown graphs}
\author{Lei Chen, Alice Devillers, Luke Morgan, Friedrich Rober}
\address{Lei Chen, Faculty of Mathematics, Bielefeld University, Bielefeld 33615, Germany.}
\address{Alice Devillers and Luke Morgan, The Centre for the Mathematics of Symmetry and Computation, The University of Western Australia, 35 Stirling Highway, Perth WA, 6009.}
\address{Friedrich Rober, Algebra and Representation Theory, RWTH Aachen University, Pontdriesch
10-16, 52062 Aachen, Germany.}

\thanks{\textbf{Acknowledgements}: We thank the Centre for the Mathematics of Symmetry and Computation at The University of
Western Australia for organising the annual research retreat of 2023 at which this research project
began. The research of Alice Devillers and of Luke Morgan was supported by the Australian Research
Council via the Discovery Project grants DP200100080 and DP230101268, respectively. The research of Lei Chen is supported by the Deutsch Forschungsgemeinschaft (DFG), German Research Foundation, Project-ID: 491392403-TRR 358. The research of Friedrich Rober is supported by the Deutscher Akademischer Austauschdienst (DAAD), German Academic Exchange Service, Project-ID: 57647563.
}

\date{\today}

\begin{abstract}
The distinguishing number of a permutation group $G\leqslant\Sym(\Omega)$ is the minimum number of colours needed to colour $\Omega$ in such a way that the only colour preserving element of $G$ is the identity. The distinguishing number of a graph is the distinguishing number of its automorphism group (as a permutation group on vertices).
We determine the distinguishing number of the complete bipartite graphs $K_{n,n}$ and the crown graphs $K_{n,n}-nK_2$, as well as the distinguishing number of some `large' subgroups of their automorphism groups, that is, the subgroups that are vertex- and edge-transitive and such that the induced action on each bipart is $\Alt(n)$ or $\Sym(n)$. 
We show that, if $G$ is a `large' group of automorphisms of $K_{n,n}$,  then $n-1\leqslant D(G) \leqslant n+1$. Similarly, if $G$ is a `large' group of automorphisms of a crown graph, then   $\lceil \sqrt{n-1}\rceil \leqslant D(G)\leqslant \lfloor \sqrt{n}\rfloor+1$.

\smallskip

\end{abstract}
\maketitle
\section{Introduction}
A vertex colouring of a graph \(\Gamma\) is a labelling of the vertices of \(\Gamma\). %, and the colouring is said to be \emph{proper} if and only if adjacent vertices are labelled with different colours. 
The \emph{distinguishing number} \(D(\Gamma)\) of a graph \(\Gamma\) is  the smallest integer \(k\) such that \(\Gamma\) has a colouring with \(k\) labels that is only preserved by the identity.  The investigation of distinguishing number has its roots in the  following problem: given a ring of seemingly identical keys,
%that open different doors,
how many colours are needed to colour the keys so that the position of each key on the ring is distinguishable? In mathematical terms, this becomes  a question about the distinguishing number of  cycles. The distinguishing number was generalised to the context of  group actions by Tymoczko \cite{Ty}: for a   group \(G\) acting on a set \(\Omega\), the distinguishing number \(D(G)\) of \(G\) is the smallest integer \(k\) such that there exists a partition \(\Pi=\{P_{1},\ldots, P_{k}\}\) of \(\Omega\) with \(G_{(\Pi)}=\bigcap_{i=1}^{k}G_{P_{i}}=1\). Note that the two definitions of distinguishing numbers \(D(\Gamma)\) and \(D(G)\) for a graph \(\Gamma\) and a group \(G\) acting on \(\Gamma\) coincide when \(G=\Aut(\Gamma)\).
%Furthermore, we define the \emph{distinguishing chromatic number \(\chi_{D}(\Gamma)\)} of \(\Gamma\) to be the smallest integer \(k\) such that there exists a labelling with \(k\) labels  that  is both distinguishing and proper.

The distinguishing number has been studied for many families of graphs, such as complete bipartite graphs, cycles, paths and trees (see \cite{CT2006} and citations thereof). Many group actions have been studied, such as the imprimitive actions \cite{P2006} and the action of the general linear group on vectors \cite{klavzar}. For primitive groups, Cameron, Neumann and Saxl \cite{CNS1994} proved that all but   finitely many groups have distinguishing number two.  Groups with certain restrictions on stabilisers or motion were considered by Conder and Tucker \cite{ConderTucker} who generalised the `motion lemma' of Russell and Sundaram \cite{RussellSundaram} to show that many graphs and maps have distinguishing number two. In fact, Conder and Tucker state that ``having distinguishing number two is a generic property''. From the point of view of global symmetry, the distinguishing number of non-bipartite $2$-arc-transitive graphs was considered in \cite{DMH2018} where it was shown that (aside from complete graphs) all but finitely many graphs have distinguishing number two. As a continuation of that project we focus on families of graphs that have large automorphism groups and unbounded distinguishing number.

In this paper, we focus on two important families of bipartite graphs, the complete bipartite graphs \(K_{n,n}\) which have automorphism group $\Sym(n)\wr \Sym(2)$ and the so-called `crown graphs' \(K_{n,n}-nK_{2}\) (graphs obtained by deleting a perfect matching from a complete bipartite graph) which have automorphism group $\Sym(n)\times \Sym(2)$.

A closely related concept to distinguishing number is that of  the \emph{distinguishing chromatic number \(\chi_{D}(\Gamma)\)} of \(\Gamma\), defined to be the smallest integer \(k\) such that there exists a vertex-colouring of $\Gamma$ with \(k\) colours  that  is both distinguishing and proper (adjacent vertices must have different colours).

It was previously shown that  \(D(K_{n,n})=n+1\) \cite[Theorem 2.4]{CT2006} and  \(\chi_{D}(K_{n,n})=2n\) \cite[Theorem 2.3]{CT2006}, and in \cite{P2009} it was proved that \(\chi_{D}(K_{n,n}-nK_{2})=\lceil 2\sqrt{n}\rceil\). We complete the results on the distinguishing number of the crown graphs by proving:

\begin{theorem}
\label{intro thm: dno of crown}
    Let $\Gamma=K_{n,n}-nK_2$. Then $D(\Gamma)=\lfloor \sqrt{n}\rfloor +1$
\end{theorem}
In  comparison with the results in \cite[Theorem 1]{P2009}, we observe that \(\chi_{D}(\Gamma) \approx  2D(\Gamma)\) for \(\Gamma=K_{n,n}-nK_{2}\), which presents a similar pattern to that of \(\Gamma=K_{n,n}\).

The above theorem is a consequence of a more detailed study on the value of \(D(G)\) where \(G\leqslant \Aut(\Gamma)\) and \(\Gamma=K_{n,n}\) or \(K_{n,n}-nK_{2}\). The groups $G$ we consider can be viewed as the `large' subgroups of $\Aut(\Gamma)$. In the theorem below, we  denote by $G^+$ the unique index two subgroup of $G$ that preserves the bipartition of $\Gamma$.
%in that {\color{red}$G$ induces  on each bipart the full alternating group}.  %{\color{red} what dos $\sim$ mean here? }{\color{blue}It means that \(\chi_{D}(\Gamma)\) is apporximately two times as \(D(\Gamma)\)}
%Our main result is presented in the following theorem.
\begin{theorem}
\label{intro thm: table values}
      Let $\Gamma=K_{n,n}$ or $K_{n,n}-nK_2$  and let $G  \leqslant \Aut(\Gamma)$ be such that $G^+$ appears in Table~\ref{tab:main table}. Then the value of  $D(G)$ is given in the corresponding column of Table~\ref{tab:main table}.
\end{theorem}

In an investigation into the distinguishing number of $2$-arc-transitive graphs \cite{CDMR}, the groups and graphs considered in Theorem~\ref{intro thm: table values} have turned out to be rather special. Our interest in them is justified by the following classification theorem. We require one further definition. Let $\Gamma$ be a bipartite graph with vertex-transitive group of automorphisms $G\leqslant \Aut(\Gamma)$, and let $\Delta$ and $\Delta'$ denote the two parts of $V\Gamma$. Letting $n$ be the size of $\Delta$ (which equals the size of $\Delta'$), $G^+$ induces a permutation group on   $\Delta$, denoted by $(G^+)^{\Delta}$, which may be viewed as a subgroup of $\Sym(n)$.
%The graphs $\Gamma$ and groups $G$ considered in Theorem~\ref{main} are justified by the following classification theorem. 

\begin{theorem}
\label{main}
      Let $\Gamma$ be a connected bipartite $G$-vertex transitive and $G$-edge transitive graph where $G\leqslant \Aut(\Gamma)$. Let $\Delta$  denote one of the two parts of $\Gamma$, of size $n$ say, and suppose that
      %Let $H=(G^+)^\Delta$.
     $\Alt(n)\leqslant (G^+)^\Delta\leqslant \Sym(n)$.  Then $\Gamma=K_{n,n}$ or $K_{n,n}-nK_2$ and, up to conjugation in  $\Sym(n)\times \Sym(n)$, $G^+$ appears in Table~\ref{tab:main table}.
\end{theorem}

\begin{table}[h]
\centering
\begin{tabular}{|c|c|c|}
\hline
 $\Gamma$ &  $G^+$  & $D(G)$ \\  
\hline 
 $K_{n,n}$ & $\Alt(n)\times \Alt(n)$ &$n-1$ \\
 $K_{n,n}$ &$\langle \Alt(n)\times \Alt(n),((1,2),(1,2))\rangle$ & $n$ \\
 $K_{n,n}$ &$\Sym(n)\times \Sym(n)$ & $n+1$ \\
\hline
 $K_{n,n}-nK_2$ &$\Diag(\Sym(n)\times \Sym(n))$ & $\lfloor\sqrt{n} \rfloor+1$ \\
 $K_{n,n}-nK_2$ &$\Diag(\Alt(n)\times \Alt(n))$ & $\lceil\sqrt{n-1} \rceil$ \\
\hline
 $K_{6,6}$ & $\Diag_\varphi(H\times H)$ &$3$ \\
\hline
$K_{4,4}$ &$\langle V_4\times V_4,\Diag(H\times H)\rangle$ &  $3$ \\
\hline
\end{tabular}
\caption{Values of $D(G)$ where  $G \leqslant \Aut(\Gamma)$,   $H=\Alt(n)$ or $\Sym(n)$ and $\varphi$ is a certain outer automorphism of $\Sym(6)$ specified in Proposition~\ref{group-cases}.}
\label{tab:main table}
\end{table}

We remark that the subgroup $G^+$ listed in Table~\ref{tab:main table} does not \emph{uniquely} determine $G$, but all possibilities are listed in Proposition~\ref{group-cases}. Further, the definition of a `diagonal group' such as $\Diag_\varphi(H\times H)$ is given in Section~\ref{sec:prelims}.
 
Our results will be applied in forthcoming work in which we investigate the distinguishing number of $2$-arc-transitive graphs \cite{CDMR}. As a result of that investigation and \cite[Corollary 5]{DMH2018}, the graph-group pairs listed in Table~\ref{tab:main table} together with the complete graphs $K_n$ are the only $2$-arc-transitive   graphs with distinguishing number \emph{not} bounded by an absolute constant. 

%{\color{red}bipartite was missing in the last sentence. Or we could leave it out and add $K_n$. add reference to DHM.}

\section{Preliminaries}
\label{sec:prelims}

Let $\Gamma$ be a connected bipartite $G$-vertex transitive graph where $G\leqslant \Aut(\Gamma)$.
Let $\Delta$ and $\Delta'$ denote the parts in the (unique) partition of $V\Gamma$ into two parts.
Since $\Gamma$ is $G$-vertex transitive, we have $|\Delta|=|\Delta'|=n$.

We label the vertices in \(\Delta\) by \(\{v_1,\ldots,v_n\}\),
and  the vertices in \(\Delta'\) by \(\{u_1,\ldots,u_n\}\).
Let $\tau$ be the  involution swapping $\Delta$ and $\Delta'$ defined by
$$
v_i^\tau=u_i
\qquad \text{ and } \qquad
u_i^\tau=v_i.
$$
Then
$
\Aut(\Gamma)
\leqslant
\bigl(\Sym(\Delta)\times \Sym(\Delta')\bigr) \rtimes \langle \tau \rangle
\cong
\Sym(n) \wr \Sym(2).
$

We write an element in $\Sym(\Delta)\times \Sym(\Delta')$ in the form $(g,g')$, where $g,g'\in \Sym(n)$ and  $(g,g')$ acts on $V\Gamma$ as follows:
$$
v_i^{(g,g')}=v_{i^g}
\qquad \text{ and } \qquad
u_i^{(g,g')}=u_{i^{g'}}
$$

Let $G^+=G_\Delta=G_{\Delta'}$, and let  $U_1=(G^+)^\Delta$, $U_2=(G^+)^{\Delta'}$, so that $G^+\leqslant U_1\times U_2$.  Then $ U_1\cong U_2$ and we identify both groups with a subgroup $H$ of $\Sym(n)$. 
%Since $G$ has to permute the parts $\Delta$ and $\Delta'$, a theorem of Kovacs allows us to assume $G \lesssim H \wr \Sym(2)=(H\times H)\rtimes\langle \tau\rangle$, where $\tau$ is an involution swapping $\Delta$ and $\Delta'$.

For $\varphi\in\Aut(H)$, we define 
$\Diag_\varphi(H\times H)$ as the subgroup $\{(h,h^\varphi) \mid h\in H\}$ of $H\times H$.
When $\varphi$ is the trivial automorphism, we simply write $\Diag(H\times H)$.

%{\color{red} not sure if we keep this lemma or not, depends a bit how we justify which groups we are considering.}

%{\color{violet} Note to myself: check left-transversals, and H* not Inn(H)!}

\begin{lemma}\label{lem:diag_rep_enough}
Let $H \leqslant \Sym(n)$ and $H^*\leqslant \Aut(H)$  denote the group of automorphisms of $H$ induced by $N_{\Sym(n)}(H)$, that is $H^*=N_{\Sym(n)}(H)/C_{\Sym(n)}(H)$.
Let $\mu, \varphi \in \Aut(H)$. If $\mu H^*=\varphi H^*$, then $\Diag_\mu(H\times H)$ and $\Diag_\varphi(H\times H)$ are conjugate in $\Sym(n)\times \Sym(n)$.  
\end{lemma}
\begin{proof}
    Suppose $\mu = \varphi c_t$ for some \(c_t\in H^*\), that is,  there exists \(t\in N_{\Sym(n)}(H)\) such that \(h^{c_t}=t^{-1}ht\) for all \(h\in H\). Then $(h,h^\mu) = (h,h^{\varphi c_t})=(h,(h^\varphi)^t)$. Since $(1,t^{-1}) \in \Sym(b)\times\Sym(b)$ we have 
\begin{eqnarray*}
    \Diag_\mu(H\times H)^{(1,t^{-1})} &= & \{ (h,h^\mu)^{(1,t^{-1})} : h\in H \} \\
    &=& \{ (h,(h^\varphi)^t)^{(1,t^{-1})} : h\in H \}\\
    &=& \{ (h,h^\varphi) : h\in H \}\\
    &=& \Diag_\varphi(H\times H) 
\end{eqnarray*}
as required.
\end{proof}

Thus, if we fix a transversal $\mathcal T = \{\varphi_1,\ldots,\varphi_n\}$ to $H^*$ in $\Aut(H)$, then for any $\varphi \in \Aut(H)$, we know $ \Diag_\varphi(H\times H)$ is conjugate to  $ \Diag_{\varphi_i}(H\times H)$ for some $i\in \{1\,\ldots,n\}$.

\medskip
In this paper we are interested in the case where $\Alt(n)\leqslant H\leqslant \Sym(n)$.

\begin{proposition}\label{group-cases}
    Let $\Gamma$ be a connected bipartite $G$-vertex transitive and $G$-edge transitive graph where $G\leqslant \Aut(\Gamma)$. Let $\Delta$ and $\Delta'$ denote the parts, each of size $n$, and let $\tau$ be an involution swapping $\Delta$ and $\Delta'$. Let $H$ be the subgroup of $\Sym(n)$ isomorphic to  $(G^+)^\Delta$.
    If $\Alt(n)\leqslant H\leqslant \Sym(n)$ then, up to conjugation by an element in  $\Sym(n)\times \Sym(n)$, one of the following holds:
    \begin{table}[ht]
        \centering
        \begin{tabular}{c|ccc}
         Case&\(\Gamma\)&\(G^+\)&\(G\)  \\
         \hline
         (a) &\(K_{n,n}\)&\(\Sym(n) \times \Sym(n)\)&\(\Sym(n)\wr\Sym(2)\)\\
         \hline
         (b)&\(K_{n,n}\)&\(\Alt(n)\times \Alt(n)\)&\(\Alt(n)\wr\Sym(2)\)\\
         \hline
         (c)&\(K_{n,n}\)&\(\bigl\langle\, \Alt(n)\times \Alt(n),~ \bigl((1,2),(1,2)\bigr) \,\bigr\rangle\)&\(G^+\rtimes \langle\, \tau \,\rangle\)\\
         &&&\(G^+\rtimes \bigl\langle\, \bigl( (1,2), \mathrm{id} \bigr) \,\tau \,\bigr\rangle\)\\
         \hline
         (d)&\(K_{n,n}-nK_2\)&\(\Diag(\Sym(n)\times\Sym(n))\)&\(G^+\rtimes\langle\tau\rangle\)\\
         \hline
         (e)&\(K_{n,n}-nK_2\)&\(\Diag(\Alt(n)\times\Alt(n))\)&\(G^+\rtimes\langle\tau\rangle\)\\
         &&&\(G^+\rtimes \langle ((1,2),(1,2))\tau\rangle\)\\
         \hline
         (f)&\(K_{6,6}\)&\(\Diag_\varphi(\Sym(6)\times\Sym(6))\)&\(G^+\rtimes\langle \tau\rangle\)\\
         \hline
         (g)&\(K_{6,6}\)&\(\Diag_\varphi(\Alt(6)\times\Alt(6))\)&\(G^+\rtimes\langle \tau\rangle\)\\
         &&&\(G^+\rtimes\langle((1,2),(1,2)^\varphi) \tau\rangle\)\\
         \hline
         (h)&\(K_{4,4}\)&\(\langle V_4\times V_4,\Diag(\Sym(4)\times\Sym(4))\rangle\)&\(G^+\rtimes\langle\tau\rangle\)\\
         \hline
         (i)&\(K_{4,4}\)&\(\langle V_4\times V_4,\Diag(\Alt(4)\times\Alt(4))\rangle\)&\(G^+\rtimes\langle\tau\rangle\)\\
         &&&\(G^+\rtimes \langle ((1,2),(1,2))\tau\rangle\)\\
         
        \end{tabular}
        \caption{\((\Gamma,G^+,G)\) for connected bipartite \(G\)-vertex transitive and \(G\)-edge transitive graph where \(\Alt(n)\leqslant H:=(G^+)^\Delta\leqslant\Sym(n).\) }
        \label{tab:placeholder}
    \end{table}
    
In cases (f) and (g), $\varphi$ is a chosen outer automorphism of $\Sym(6)$ of order $2$.
Moreover, $H=\Sym(n)$ in cases (a), (c), (d), (f), and $H=\Alt(n)$ in cases (b), (e), (g).

\end{proposition}
%{\color{red}
%This is actually not true when $n=4$. So do we want to do this for $n\geq 5$ or not? What do we need for the other paper?

%When $n=4$ we get another 3 groups
% \begin{enumerate}
  % \item[(h)] $G^+=\langle V_4\times V_4,\Diag(\Sym(4)\times %\Sym(4))\rangle$ and $G= G^+\rtimes \langle\tau\rangle$ %TransitiveGroup(8,41)
   % \item[(i)] $G^+=\langle V_4\times V_4,\Diag(\Alt(4)\times %\Alt(4))\rangle$ and (1) $G= G^+\rtimes \langle\tau\rangle$ %TransitiveGroup(8,33) 
     %or (2) $G= G^+\rtimes \langle((1,2),(1,2))\tau\rangle$. %TransitiveGroup(8,34)

    % In cases (h)--(i)  $\Gamma=K_{n,n}$.
                  
    %   \end{enumerate}

%}

\begin{remark} Before we get to the proof, note that the embedding theorem (see \cite[Theorem 5.5]{PS}) states that,  up to conjugation   in $\Sym(2n)$,   $G \leqslant H \wr \Sym(2)=(H\times H)\rtimes\langle \tau\rangle$. This is not the route we chose here, but we show here, in each case, which subgroup of $H \wr \Sym(2)$ {\color{red}the group  $G$ is} conjugate to.
First note that $G\leqslant  H \wr \Sym(2)$ in all cases except the second line of cases (e), (g), and (i). 

In these three cases, $H=\Alt(n)$ and   $G\not\leqslant  H \wr \Sym(2)$.
However $G^{(\mathrm{id},(1,2))}\leqslant  H \wr \Sym(2)$.

Indeed it is easy to check that for $G$ as in the second line of (e) 
$$G^{(\mathrm{id},(1,2))}=\Diag_\mu(\Alt(n)\times \Alt(n))\rtimes \langle\tau\rangle$$
where $\mu$ is the automorphism of $H$ induced by conjugation by $(1,2)$.

For $G$ as in the second line of  (g) we have 
$$G^{(\mathrm{id},(1,2))}=\Diag_{\varphi\mu}(\Alt(n)\times \Alt(n))\rtimes \langle(\mathrm{id},(1,2)(1,2)^\varphi)\tau\rangle,$$
and $(\mathrm{id},(1,2)(1,2)^\varphi)\tau\in \Alt(n)\wr \Sym(2)$ since $(1,2)$ and $(1,2)^\varphi$ are both odd permutations.

Finally for $G$ as in the second line of (i), 
$$G^{(\mathrm{id},(1,2))}=\langle V_4\times V_4,\Diag_\mu(\Alt(4)\times \Alt(4))\rangle \rtimes \langle\tau\rangle.$$

\end{remark}

\begin{proof}
It is straightforward to check computationally that 
\begin{itemize}
\item one of (a) or the second line of case (c) holds if $n=2$. Note that first line of case (c) (which is equal to (d))  does not yield a graph $\Gamma$ satisfying the required conditions. 
   \item  one of (a)--(e) holds if $n=3$,  but  case of (e)(2)  does not yield any graph $\Gamma$ satisfying the required conditions.
   \item one of (a)--(e), (h)--(i) holds if $n=4$. 
\end{itemize}

From now we assume that $n\geq 5$, so that $\Alt(n)$ is simple and $H$ has trivial centre.

Let $K_1$ be the kernel of the action of $G^+$ on $\Delta$ and $K_2$ be the kernel of the action of $G^+$ on $\Delta'$. By definition $K_1\cap K_2=1$, so $K_1\times K_2\leqslant G^+.$ Note that $K_1\cong K_2$ since $G$ is transitive on  $\Delta\cup\Delta'$.

Assume $K_1\neq 1$. Since $K_1\lhd (G^+)^{\Delta'}$, and $(G^+)^{\Delta'}$ is primitive on $\Delta'$, it follows that $K_1$ is a transitive normal subgroup of $(G^+)^{\Delta'}$. %Therefore $\Gamma=K_{n,n}.$
We claim that one of (a), (b), (c) holds.

Note $K_1$ contains $\soc((G^+)^{\Delta'})\cong \Alt(n)$, so $K_1\cong \Sym(n)$ or $K_1\cong \Alt(n)$.

Assume first that $K_1\cong \Sym(n)$. Then $G^+=\Sym(n) \times \Sym(n)$ and (a) holds since $G$ is transitive.

Now assume that  $K_1\cong \Alt(n)$, so $\Alt(n) \times \Alt(n)\leqslant G^+<\Sym(n) \times \Sym(n)$.

If $G^+=\Alt(n) \times \Alt(n)$, then $G=G^+\rtimes \langle(a,b)\tau\rangle$ for some $a,b\in\Sym(n)$. It follows that $((a,b)\tau)^2=(ab,ba)\in G^+$, and thus $a,b$ have the same parity. If they are both even then $(a,b)\in G^+$ and so  $G=G^+\rtimes \langle(a,b)\tau\rangle=G^+\rtimes \langle \tau\rangle$. If they are both odd then $(a,b)((1,2),(1,2))\in G^+$, $G=G^+\rtimes \langle(a,b)\tau\rangle=G^+\rtimes \langle ((1,2),(1,2)) \tau\rangle$. Finally notice that $\tau^{((1,2),\mathrm{id})}=((1,2),(1,2)) \tau$, so that $G^+\rtimes \langle \tau\rangle$ and $G^+\rtimes \langle ((1,2),(1,2)) \tau\rangle$ are conjugate under $\Sym(n) \times \Sym(n)$. Thus (b) holds.
%{\color{blue} not sure if we need these details but I wanted to work them out to be 100\% sure.}

Otherwise,  the fact that $(G^+)^{\Delta}\cong (G^+)^{\Delta'}$ implies that $G^+=\langle\Alt(n)\times \Alt(n),((1,2),(1,2))\rangle.$ 
Again $G=G^+\rtimes \langle(a,b)\tau\rangle$ for some $a,b\in\Sym(n)$. If $a$ and $b$ have the same parity then $(a,b)\in G^+$ and so  $G=G^+\rtimes \langle(a,b)\tau\rangle=G^+\rtimes \langle \tau\rangle$. If $a$ and $b$ have different parity then $(a,b)((1,2),\mathrm{id})\in G^+$, $G=G^+\rtimes \langle(a,b)\tau\rangle=G^+\rtimes \langle ((1,2),\mathrm{id}) \tau\rangle$. Thus (c) holds.

This concludes the case $K_1\neq 1$.

Now assume $K_1=1$, that is $G^+$ acts faithfully on $\Delta$ (and hence also on $\Delta'$). 
    For each $g\in G^+$, we may write $g=(g^\Delta, g^{\Delta'}) \in H \times H$ and $g^{\Delta}=1$ if and only if $g^{\Delta'}=1$. Since $H=(G^+)^{\Delta}$, for all $h\in H$ there is a unique $g_h\in G^+$ such that $(g_h)^{\Delta}=h$. Define $\varphi :H \rightarrow H$ by declaring $\varphi :h \mapsto (g_h)^{\Delta'}$ so we may write $g_h = (h,h^\varphi)$. Since $g^\Delta = 1$ if and only if $g^{\Delta'}=1$ we have $h^\varphi =1$ if and only if $h=1$. For $h,k\in H$ we have
\[
g_h g_k = (h,h^\varphi)(k,k^\varphi) =( hk , h^\varphi k^\varphi) = (hk, (hk)^\varphi)
\]
so that $h^\varphi k^\varphi = (hk)^{\varphi}$. Hence $\varphi \in \Aut(H)$ and $G^+ = \{ (h,h^\varphi) \mid h\in H\}=\Diag_\varphi(H\times H)$.
%{\color{red} could replace this argument with a reference? See also \cite[Lemmas 3.6, 3.4]{P2003}. }

Now we use Lemma \ref{lem:diag_rep_enough}. %{\color{purple}Let us split the analysis into two cases: (i) $H^*= \Aut(H)$ and (ii) $H^*\neq \Aut(H)$.

%\medskip\noindent
%\emph{Case (i): $H^*= \Aut(H)$}

   If $H^*= \Aut(H)$, then up to conjugation in  $\Sym(n)\times \Sym(n)$, we may assume that $G^+=\Diag(H\times H)$.
  If $H^*\neq \Aut(H)$, then 
$n=6$ and  $H^*$ has index $2$ in $\Aut(H)$, so up to conjugation  in  $\Sym(n)\times \Sym(n)$, we have that  $G^+=\Diag(H\times H)$ or $G^+=\Diag_\varphi(H\times H)$, where $\varphi$ is a chosen outer automorphism of $\Sym(6)$ of order 2. We examine the two cases for $G^+$ separately.

First we assume that  $G^+=\Diag(H\times H)$.
Then $G=G^+\rtimes \langle (a,b)\tau\rangle$ for some $a,b\in \Sym(n)$.
For any $(h,h)\in G^+$ we have that $(h,h)^{(a,b)\tau}=(h^b,h^a)\in G^+=\Diag(H\times H)$. Thus $ba^{-1}$ centralises $H$, and so $a=b$ since $Z(H)$ is trivial.

If $H=\Sym(n)$, then $(a,a)\in G^+$, and so  $G=G^+\rtimes \langle (a,a)\tau\rangle=G^+\rtimes \langle \tau\rangle$. Thus (d) holds.

Suppose  $H=\Alt(n)$. Then we have two possibilities.
When $a$ is even, then  $(a,a)\in G^+$, and so  $G=G^+\rtimes \langle (a,a)\tau\rangle=G^+\rtimes \langle\tau\rangle$. 
When $a$ is odd, then  $(a,a)((1,2),(1,2))\in G^+$, and so  $G=G^+\rtimes \langle (a,a)\tau\rangle=G^+\rtimes \langle((1,2),(1,2))\tau\rangle$. Thus (e) holds.

%Now assume $n=3$ and $H=\Alt(3)$ and $a\neq b$. The condition $ba^{-1}$ centralises $H$ implies that $ba^{-1}\in\Alt(3)$. On the other hand $((a,b)\tau)^2=(ab,ba)\in \Diag(H\times H)$, so $a,b$ commute. It easily follows that $a,b\in H=\Alt(3)$. Thus $G=\langle \Diag(H\times H),(a,b)\tau\rangle=\langle \Diag(H\times H),(1,ba^{-1})\tau\rangle$.  But then $G$ is conjugate in $\Sym(3)\times \Sym(3)$ to $\langle\Diag(H\times H),\tau\rangle$ (by $(1,ab^{-1})$).
%{\color{purple} \medskip\noindent
% \emph{Case (ii): $H^*\neq \Aut(H)$}

% If $H^*\neq \Aut(H)$, then 
%$n=6$ and  $H^*$ has index $2$ in $\Aut(H)$, so up to conjugation  in  $\Sym(n)\times \Sym(n)$, we have that  $G^+=\Diag(H\times H)$ or $G^+=\Diag_\varphi(H\times H)$, where $\varphi$ is a chosen outer automorphism of $\Sym(6)$ of order 2. }
Finally assume that $G^+=\Diag_\varphi(H\times H)$, where $\varphi$ is a chosen outer automorphism of $\Sym(6)$ of order 2.
Then $G=G^+\rtimes \langle (a,b)\tau\rangle$ for some $a,b\in \Sym(n)$.
For any $(h,h^\varphi)\in G^+$ we have that $(h,h^\varphi)^{(a,b)\tau}=((h^\varphi)^b,h^a)\in G^+$. Thus $h^a=((h^\varphi)^b)^\varphi=h^{b^\varphi}$ (since we assumed $\varphi$ has order 2). It follows that $a=b^\varphi$ (or equivalently, that  $b=a^\varphi$) since $Z(H)$ is trivial.
Thus $G=G^+\rtimes \langle (a,a^\varphi)\tau\rangle$.

If $H=\Sym(n)$, then $(a,a^\varphi)\in G^+$, and so  $G=G^+\rtimes \langle (a,a^\varphi)\tau\rangle=G^+\rtimes \langle \tau\rangle$. Thus (f) holds.

Suppose  $H=\Alt(n)$. Then we have two possibilities.
When $a$ is even, then  $(a,a^\varphi)\in G^+$, and so  $G=G^+\rtimes \langle (a,a^\varphi)\tau\rangle=G^+\rtimes \langle\tau\rangle$. 
When $a$ is odd, then  $(a,a^\varphi)((1,2),(1,2)^\varphi)\in G^+$, and so  $G=G^+\rtimes \langle (a,a^\varphi)\tau\rangle=G^+\rtimes \langle((1,2),(1,2)^\varphi)\tau\rangle$. Thus (g) holds.

We now prove the assertions about the graph $\Gamma$. We assumed that the graph is connected, so there exists an edge $e=\{v_i,u_j\}$ of $\Gamma$

In cases (a)--(c), we have seen that $K_1$ is a transitive normal subgroup of $(G^+)^{\Delta'}$. Therefore every vertex in $\Delta'$ is adjacent to $v_i$, and by vertex-transitivity we get that $\Gamma=K_{n,n}.$

Assume one of cases (d)--(e) holds. Consider the edge $e=\{v_i,u_j\}$.  If $i=j$, then the orbit of this edge under $G$ is simply the set of pairs $\{v_t,u_t\}$ for $t=1,2,\ldots,n$. Since $\Gamma$ is $G$-edge transitive, it follows that $\Gamma$ is not connected. Thus  $i\neq j$.

%Unless $H=\Alt(3)$, w
Since $H$ is 2-transitive,   there is a permutation in $H$ fixing $i$ and mapping $j$ to any element different from $i$. Thus all edges $\{v_i,u_t\}$ where $t\neq i$ are in the $G$-orbit of $e$.  By vertex-transitivity, the $G$-orbit of $e$ consists exactly of  all pairs $\{v_k,u_\ell\}$ where $k\neq \ell$. Since $\Gamma$ is $G$-edge transitive, we get that $\Gamma=K_{n,n}-nK_2$.
%Now if $H=\Alt(3)$, $G^+$ has two orbits on pairs $\{v_k,u_\ell\}$ where $k\neq \ell$, which must be fused under $G$ since $\Gamma$ is connected. Thus again  $\Gamma=K_{n,n}-nK_2$.
%{\color{red} this is not completely obvious  that the pairs $\{v_i,u_i\}$ are not in the same orbit using an element of $G\setminus G^+$. We might need to go into more details here of what an element in $G\setminus G^+$ can be (by the way I think there is a mistake in that theorem in the main paper). When $n\le 4$ and $n=3$ with $H=\Sym(3)$, it is easy to show it must be $(x,x)\tau$. When $n=3$ and $H=\Alt(3)$ it can be $((1,2,3),id)\tau$ but then $G$ is conjugate in $\Sym(3)\times \Sym(3)$ to $\langle\Diag(H\times H),\tau\rangle$.   }

Assume one of cases (f)--(g) holds.
Consider the edge $e=\{v_i,u_j\}$. 
    Point stabilisers in $H$ and in $H^\varphi$ are not conjugate. Note that \(\Alt(5)\leqslant H_{v_{i}}\leqslant \Sym(5)\) and so \(\PSL_{2}(5)\leqslant H_{v_{i}}^\varphi\leqslant \PGL_{2}(5)\). In particular $H_{v_i}^\varphi$ is transitive on $\Delta'$. It follows that $\Gamma=K_{6,6}$.

 Finally assume one of cases (h)--(i) holds.
Then $K_1$ is $V_4$ acting transitively on $\Delta'$. It follows that $\Gamma=K_{4,4}$.
\end{proof}

\section{Complete bipartite graphs}

In this section we compute the distinguishing numbers of some families of groups that act $2$-arc-transitively on a complete bipartite graph $K_{n,n}$. 

The following proposition tackles cases (a)--(c) of Proposition \ref{group-cases}.

\begin{proposition}\label{prop:notfaithful}
Let $\Gamma = K_{n, n}$ with $n\geqslant 2$ and suppose $\Alt(n)\times \Alt(n)\leqslant G \leqslant \Sym(n) \wr \Sym(2) = \Aut(\Gamma)$ with $G$ transitive on $V\Gamma$. Then 
$$D(G) =\begin{cases}
    n-1&\text{ if }G^+=\Alt(n)\times \Alt(n)\\
   n& \text{ if }G^+=\langle \Alt(n)\times \Alt(n),((1,2),(1,2))\rangle\\
    n+1&\text{ if }G^+=\Sym(n)\times \Sym(n)\\
\end{cases} $$
\end{proposition}

\begin{proof}

Let \(k:=D(G)\) and \(\Sigma=\{\Sigma_{1},\ldots, \Sigma_{k}\}\) be a distinguishing partition of \(V\Gamma\), so that \(G_{(\Sigma)}=1\).

If $|\Sigma_i\cap \Delta|> 2$ for some $i$, say $v_a,v_b,v_c\in \Sigma_i\cap \Delta$, then $g=((a,b,c),\mathrm{id})\in  G_{(\Sigma)}$, a contradiction.
So $|\Sigma_i\cap \Delta|\leqslant 2$ for each $i$.
Suppose there are two parts $\Sigma_i,\Sigma_j$ containing two points of $\Delta$ each. Say $v_a,v_b\in \Sigma_i\cap \Delta, v_c,v_d\in \Sigma_j\cap \Delta$. Then $g=((a,b)(c,d),\mathrm{id})\in  G_{(\Sigma)}$, a contradiction.
So all parts $\Sigma_i$ contain at most one point of $\Delta$, except possibly one part that contains two points. Thus $k\geqslant n-1$.

Assume first that $G^+=\Alt(n)\times \Alt(n)$, that is $G=\Alt(n)\wr \Sym(2)=\langle \Alt(n)\times \Alt(n),\tau\rangle$ by  Proposition \ref{group-cases}. Note that $n\geqslant 3$ since $G$ is transitive.

We will show that \(k\leqslant n-1\) by displaying a partition \(\mathcal{P}=\{P_1,\ldots, P_{n-1}\}\)  with trivial stabiliser \(G_{(\mathcal{P})}\). %This in turn will show that \(k=n-1\).

We take
\begin{equation*} 
P_{i}  \coloneqq
\begin{cases}
  \{v_1,v_2,u_1\}& \text{if } i=1\\ 
   \{v_3,u_2,u_3\}& \text{if } i=2\\ 
  \{u_{i+1},v_{i+1}\} & \text{if } 3\leqslant i\leqslant n-1.
\end{cases}
\end{equation*}

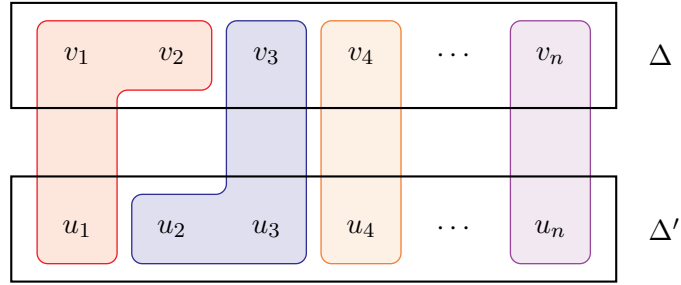
\begin{center}
\begin{tikzpicture}
    \drawNodes
    \begin{pgfonlayer}{bg}
    % Partition 1
    \draw [line width = \thinLine, color=Red, fill=Red!10, rounded corners]
    ($ (v_1.north west) + (-\extH, \extV) $) --
    ($ (v_2.north east) + (\extH, \extV) $) --
    ($ (v_2.south east) + (\extH, -\extV) $) --
    ($ (v_1.south east) + (\extH, -\extV) $) --
    ($ (u_1.south east) + (\extH, -\extV) $) --
    ($ (u_1.south west) + (-\extH, -\extV) $) --
    cycle;
    % Partition 2
    \draw [line width = \thinLine, color=Blue, fill=Blue!10, rounded corners]
    ($ (v_3.north west) + (-\extH, \extV) $) --
    ($ (v_3.north east) + (\extH, \extV) $) --
    ($ (u_3.south east) + (\extH, -\extV) $) --
    ($ (u_2.south west) + (-\extH, -\extV) $) --
    ($ (u_2.north west) + (-\extH, \extV) $) --
    ($ (u_3.north west) + (-\extH, \extV) $) --
    cycle;
    % Partition 3
    \draw [line width = \thinLine, color=Orange, fill=Orange!10, rounded corners]
    ($ (v_4.north west) + (-\extH, \extV) $) --
    ($ (v_4.north east) + (\extH, \extV) $) --
    ($ (u_4.south east) + (\extH, -\extV) $) --
    ($ (u_4.south west) + (-\extH, -\extV) $) --
    cycle;
     % Partition n - 1
    \draw [line width = \thinLine, color=Purple, fill=Purple!10, rounded corners]
    ($ (v_n.north west) + (-\extH, \extV) $) --
    ($ (v_n.north east) + (\extH, \extV) $) --
    ($ (u_n.south east) + (\extH, -\extV) $) --
    ($ (u_n.south west) + (-\extH, -\extV) $) --
    cycle;
    \end{pgfonlayer}
\end{tikzpicture}
\end{center}

Let \(g\in G_{(\mathcal{P})}\). Then $g\in G^+=\Alt(n)\times \Alt(n)$ since $|P_1\cap\Delta|\neq |P_1\cap\Delta'|$.
Thus $g$ fixes $P_i\cap \Delta$ and $P_i\cap \Delta'$ for each $i$. In particular $g$ fixes the vertices $v_i$ for $i\geqslant 3$ and $u_j$ for $j=1$ and $j\geqslant 4$.
Thus $g\in\langle ((1,2),\mathrm{id}),(\mathrm{id},(2,3))\rangle$. This subgroup intersects $\Alt(n)\times \Alt(n)$ trivially, and so $g=1$. It follows that $G_{(\mathcal{P})}=1$ and so $k\leqslant n-1$. Therefore $k=n-1$.

Assume now that $G^+=\langle \Alt(n)\times \Alt(n),((1,2),(1,2))\rangle$, so that by  Proposition \ref{group-cases} $G=\langle \Alt(n)\times \Alt(n),((1,2),(1,2)),\tau\rangle$ or $G=\langle \Alt(n)\times \Alt(n),((1,2),\mathrm{id})\tau\rangle$.

Recall that $k\geqslant n-1$.
Suppose $k=n-1$. 
Then all parts of $\Sigma$ intersect $\Delta$ in exactly one point, except for one, and similarly for $\Delta'$.  Say $v_a,v_b\in \Sigma_i\cap \Delta$ and   $u_c,u_d\in \Sigma_j\cap \Delta'$ (where $i$ could  be equal to $j$). Then $g=((a,b),(c,d))\in G^+$ since both entries have the same parity, and thus $g\in G_{(\Sigma)}$, a contradiction. Therefore $k\geqslant n$.

We will show that \(k\leqslant n\) by displaying a partition \(\mathcal{P}\) of size \(n\) with trivial stabiliser \(G_{(\mathcal{P})}\). %This in turn will show that \(k=n\).

Let $\mathcal P = \{P_1,\ldots,P_{n}\}$ be as below and let \(g\in G_{(\mathcal{P})}\).
\begin{equation*} 
P_{i}  \coloneqq
\begin{cases}
  \{v_1,v_2,u_1\}& \text{if } i=1\\ 
   \{u_2\}& \text{if } i=2\\ 
  \{u_i,v_i\} & \text{if } 3\leqslant i\leqslant n
\end{cases}
\end{equation*}

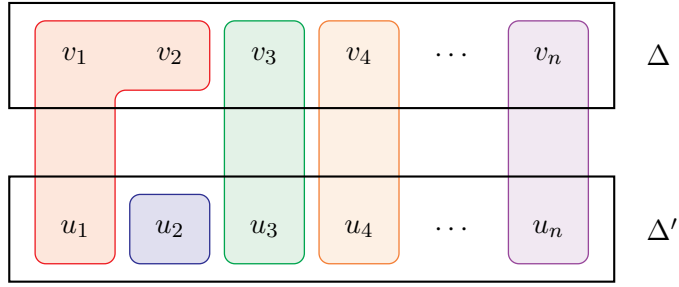
\begin{center}
\begin{tikzpicture}
    \drawNodes
    \begin{pgfonlayer}{bg}
    % Partition 1
    \draw [line width = \thinLine, color=Red, fill=Red!10, rounded corners]
    ($ (v_1.north west) + (-\extH, \extV) $) --
    ($ (v_2.north east) + (\extH, \extV) $) --
    ($ (v_2.south east) + (\extH, -\extV) $) --
    ($ (v_1.south east) + (\extH, -\extV) $) --
    ($ (u_1.south east) + (\extH, -\extV) $) --
    ($ (u_1.south west) + (-\extH, -\extV) $) --
    cycle;
    % Partition 2
    \draw [line width = \thinLine, color=Blue, fill=Blue!10, rounded corners]
    ($ (u_2.north west) + (-\extH, \extV) $) --
    ($ (u_2.north east) + (\extH, \extV) $) --
    ($ (u_2.south east) + (\extH, -\extV) $) --
    ($ (u_2.south west) + (-\extH, -\extV) $) --
    cycle;
    % Partition 3
    \draw [line width = \thinLine, color=Green, fill=Green!10, rounded corners]
    ($ (v_3.north west) + (-\extH, \extV) $) --
    ($ (v_3.north east) + (\extH, \extV) $) --
    ($ (u_3.south east) + (\extH, -\extV) $) --
    ($ (u_3.south west) + (-\extH, -\extV) $) --
    cycle;
    cycle;
    % Partition 4
    \draw [line width = \thinLine, color=Orange, fill=Orange!10, rounded corners]
    ($ (v_4.north west) + (-\extH, \extV) $) --
    ($ (v_4.north east) + (\extH, \extV) $) --
    ($ (u_4.south east) + (\extH, -\extV) $) --
    ($ (u_4.south west) + (-\extH, -\extV) $) --
    cycle;
     % Partition n
    \draw [line width = \thinLine, color=Purple, fill=Purple!10, rounded corners]
    ($ (v_n.north west) + (-\extH, \extV) $) --
    ($ (v_n.north east) + (\extH, \extV) $) --
    ($ (u_n.south east) + (\extH, -\extV) $) --
    ($ (u_n.south west) + (-\extH, -\extV) $) --
    cycle;
    \end{pgfonlayer}
\end{tikzpicture}
\end{center}

Then $g\in G^+=\langle \Alt(n)\times \Alt(n),((1,2),(1,2))\rangle$ since $P_2\subset \Delta'$.
Thus $g$ fixes $P_i\cap \Delta$ and $P_i\cap \Delta'$ for each $i$. In particular $g$ fixes the vertices $v_i$ for $i\geqslant 3$ and $u_j$ for all $j$.
Thus $g\in\langle ((1,2),\mathrm{id})\rangle$. This subgroup intersects $G^+$ trivially, and so $g=1$. It follows that $G_{(\mathcal{P})}=1$ and so $k\leqslant n$.
Thus $D(G) = n$

Finally assume that $G^+=\Sym(n)\times \Sym(n)$. It follows that   $G=\Sym(n)\wr \Sym(2) = \Aut(K_{n,n})$. Hence $D(G)=n+1$ by \cite[Theorem 2.4]{CT2006} and the proof is complete.

% {\color{red}Luke: ``what is the rest of this proof?''}
% %\subfile{../tikz/complete_bipartite_graph/An}

% Let $B = G \cap (\Sym(n) \times \Sym(n))$ be the base group of $G$
% and $\sigma = (1,2) \in B$ be the swap.
% Then $B$ is a union of right cosets of $\Alt(n)^2$ in $\Sym(n)^2$
% and must contain a non-trivial coset representative.
% Suppose for example that $x = ((1,2), ()) \in B$. Then $x^\sigma = ((), (1,2)) \in B$ and thus $B = \Sym(n)^2$, a contradiction.
% Thus $x = ((1,2), (1,2)) \in B$ and we have $G = ((\Alt(n) \times \Alt(n)).\Sym(2)).\Sym(2)$.

% Clearly we have $n \leqslant D(G) \leqslant n + 1$.
% We define a partition that shows that $D(G) = n$.
% Let \(\mathcal{P}:=\{P_{1},\ldots, P_n\}\) be a partition of \(V\Gamma\) defined by
% \[
% P_i \coloneqq
% \begin{cases}
%     \{1, 1', 2\} ,& \text{if } i = 1,\\ 
%     \{2'\} ,& \text{if } i = 2,\\ 
%     \{i, i'\} ,& \text{if } 3 \leqslant i \leqslant n.\\ 
% \end{cases}
% \]

% \vspace*{2em}
% %\subfile{../tikz/complete_bipartite_graph/InBetween}
\end{proof}

% Let $g \in G_{\mathcal{P}}$. Since $|P_1 \cap \Delta| = 2 \neq 1 = |P_1 \cap \Delta'|$, we have $g \in B$. Since $P_1^{a x} \neq P_1$ for all $a \in \Alt(n)$, we have $g \in \Alt(n)^2$. By the same arguments as in the proof of the above lemma, we have $g = 1$.

The following propositions tackle cases (f)--(i) of Proposition \ref{group-cases}.

\begin{proposition}\label{prop:DGdiag-outer}
Let \(\Gamma:=K_{6,6}\) and assume \(G^+= \Diag_\varphi(H \times H)\)  where $H=\Alt(6)$ or $\Sym(6)$ and $\varphi$ is a chosen outer automorphism of $\Sym(6)$ of order $2$. Then \(D(G)=3\).
\end{proposition}
\begin{proof}
A {\sc Magma} computation confirms this result.
\end{proof}

% {\color{red} Old Version did not cover (g)(2). Friedrich, can you check again that case (g)(2) also has $D(G)=3$?

% \begin{proposition}\label{prop:DGdiag-outer}
% Let \(\Gamma:=K_{6,6}\) and \(G:=\langle \Diag_\varphi(H \times H),\tau\rangle\)  where $H=\Alt(6)$ or $\Sym(6)$ and $\varphi$ is an outer automorphism of $\Sym(6)$. Then \(D(G)=3\).
% \end{proposition} 
% }

\begin{proposition}\label{prop:n=4}
Let \(\Gamma:=K_{4,4}\) and \(G^+=\langle V_4\times V_4, \Diag(H \times H)\rangle\)  where $H=\Alt(4)$ or $\Sym(4)$. Then \(D(G)=3\).
\end{proposition} 
\begin{proof}

Assume that there exists a distinguishing partition $ \Sigma = \{\Sigma_1,\Sigma_2\}$ of size $2$.
If $|\Sigma_1\cap \Delta|=4$, then $g=((1,2)(3,4),\mathrm{id})\in G_{(\Sigma)}$, a contradiction.

If $|\Sigma_1\cap \Delta|=2$, say $\Sigma_1\cap \Delta=\{v_i,v_j\}$. Then $\Sigma_2\cap \Delta=\{v_k,v_\ell\}$ where $\{k,\ell\}=\{1,2,3,4\}\setminus \{i,j\}$. Thus $g=((i,j)(k,\ell),\mathrm{id})\in G_{(\Sigma)}$, a contradiction.

The same argument holds for $\Delta'$. Thus (after swapping $\Sigma_1$ and $\Sigma_2$ if necessary) we may assume that $|\Sigma_1\cap \Delta|=3$ and $|\Sigma_2\cap \Delta| =1$.
%the four vertices of $\Delta$ (resp. $\Delta'$) are split $1-3$ between $\Sigma_1$ and $\Sigma_2$.
%Without loss of generality,
Say  $\Sigma_1\cap \Delta=\{v_i,v_j,v_k\}$. Suppose  $|\Sigma_1\cap \Delta'|=3$. If $\Sigma_1\cap \Delta'=\{u_i,u_j,u_k\}$,  $g=((i,j,k),(i,j,k))\in G_{(\Sigma)}$, a contradiction. Thus up to relabelling $\Sigma_1\cap \Delta'=\{u_i,u_j,u_\ell\}$. Now 
$g=((i,j,k),(i,\ell,j))=((i,j,k),(i,j,k))(\mathrm{id},(i,k)(j,\ell))\in G_{(\Sigma)}$, a contradiction.

Thus $|\Sigma_1\cap \Delta'|=1$. If $\Sigma_1\cap \Delta'=\{u_\ell\}$,  $g=((i,j,k),(i,j,k))\in G_{(\Sigma)}$, a contradiction. 
 Thus up to relabelling $\Sigma_1\cap \Delta'=\{u_i\}$. Now 
$g=((i,j,k),(j,\ell,k))=((i,j,k),(i,j,k))(\mathrm{id},(i,j)(k,\ell))\in G_{(\Sigma)}$, a contradiction.
This final contradiction shows that $\Sigma$ cannot be a distinguishing partition.

Let $\mathcal P = \{P_1,P_2,P_{3}\}$, where $P_1=\{v_1,v_2,u_1\}$, $P_2=\{v_3,u_2,u_3\}$, $P_3=\{v_4,u_4\}$.
This is obviously a partition of $V\Gamma$.

Assume \(g\in G_{(\mathcal{P})}\). Then $g\in G^+$ since $|P_1\cap \Delta|\neq |P_1\cap \Delta'|$. 
Moreover $g$ fixes $v_3=P_2\cap \Delta$, $v_4=P_4\cap \Delta$, $u_1=P_1\cap \Delta'$, $u_4=P_3\cap \Delta'$.
Note that $g^\Delta$ and $g^{\Delta'}$ have the same parity by definition of $G^+$, so $g\in\langle ((1,2),(2,3))\rangle$, but $((1,2),(2,3))\notin G^+$. Therefore  $g=(\mathrm{id},\mathrm{id}),$ and  $G_{(\mathcal{P})}$ is trivial.  Thus $\mathcal{P}$ is a distinguishing partition of size $3$, and so $D(G)=3.$

\end{proof}

\section{Crown graphs}

In this section we compute the distinguishing numbers of some families of groups that act $2$-arc-transitively on a crown graph $\Gamma=K_{n,n}-nK_2$, that is  cases (d)--(e) of Proposition \ref{group-cases}. 
Note the non-edges of this graph are the pairs $\{v_i,u_i\}$, since both $\tau$ and $((1,2),(1,2))\tau$ preserve this set of pairs.

We have  $G^+=  \Diag(H\times H)$ where $H=\Alt(n)$ or $H=\Sym(n)$. 
% (the justification for this will become apparent). 
% changing this so to not forward reference
%By Lemma \ref{lem:quot to alt or sym}, we need to consider two groups: $G=\langle  \Diag(H\times H),\tau\rangle$ where $H=\Alt(n)$ or $H=\Sym(n)$. 

\medskip
To display distinguishing partitions we will use the following  sets of vertices.
For \(n, i, q \in \mathbb{N}\) we define 
 \begin{equation*}
\mathcal{V}_n(i,q)\coloneqq\{v_{j}: (i-1)\,q < j \leqslant \min(iq,n)\}
 \end{equation*}
and
 \begin{equation*}
\mathcal{U}_n(i,q)\coloneqq\{u_{j}: j \equiv i \;(\bmod\; q),\; j\leqslant n\}.
 \end{equation*}
 We use \(\mathcal{V}_n(i, q)\) to partition \(\Delta\) by sequentially inserting up to \(q\) consecutive points into a block and \(\mathcal{U}_n(i, q)\) to partition \(\Delta'\) by sequentially inserting points at ``distance \(q\)" into a block.

% {\color{purple} Add footnote that representatives mod q are 1 ... q. Isn't this obvious from the next lemma?}

The following lemma is easy to check.

\begin{lemma}
\label{V is partition}
Let \(n,q \in \N\) with $q\leqslant n$.
\begin{enumerate}
    \item[(a)] For $1\leqslant i\leqslant q$, $\mathcal{U}_n(i,q)$ is non-empty, and \(
    \{\mathcal{U}_n(i, q) : 1 \leqslant i \leqslant q\}
\)
is a partition of $\Delta'$.
\item[(b)]  For any $i$ such that \((i-1)q < n\), $\mathcal{V}_n(i,q)$ is non-empty, and 
\(
    \{\mathcal{V}_n(i, q) : 1 \leqslant i \leqslant d\}
\)
is a partition of $\Delta$ if \((d-1)q<n\leqslant d \cdot q \).
\end{enumerate}
\end{lemma}

\begin{lemma}\label{lem:partsaresubsets}  Let \(n,q \in \N\) with $q\leqslant n$. Let $\Gamma=K_{n,n}-nK_2$ and assume  $G^+= \Diag(H\times H)$ where $H\leq\Sym(n)$. 

Suppose \(\mathcal{P}:=\{P_{1},\ldots, P_{m}\}\) is a partition of the vertex-set of $\Gamma$ %=K_{n,n}-nK_2$ 
such that for each $i$
\[P_i\cap\Delta\text{ is a subset of }\mathcal{V}_n(a, q)\text{ for some }a, \]
\[P_i\cap\Delta'\text{ is a subset of }\mathcal{U}_n(b, q)\text{ for some }b. \]
Then $G^+_{(\mathcal{P})}$ is trivial.    
\end{lemma}
\begin{proof}
Let $g\in G^+_{(\mathcal{P})}$.
Then $g=(h,h)$ for some $h\in \Sym(n)$. 
  Let $1\leqslant s\leqslant n$.  Since \(g\in G_{(\mathcal{P})}\), $v_s^g=v_{s^h}$ and $v_s$ are in the same part and hence in the same set $\mathcal{V}_n(a, q)$. In particular $ |s-s^h|<q$. On the other hand, $u_s^g=u_{s^h}$ and $u_s$ are also in the same part and hence in the same set $\mathcal{U}_n(b, q)$. Therefore $s\equiv s^h \;\pmod q$. It follows that $s^h=s$, and so $h=\mathrm{id}$, and  $g=(\mathrm{id},\mathrm{id}).$
  Therefore $G^+_{(\mathcal{P})}$ is trivial.   
  \end{proof}
%\begin{lemma}
%\label{U is partition}
%Let \(n \in \N\) and \(d, q \in \N\). If \(d = q\), then
%\[
%    \{\mathcal{U}_n(i, q) : 1 \leq i \leq d\}
%\]
%is a partition of $\{u_1, \dots, u_n\}$.
%\end{lemma}

%{\color{purple} I guess we need $n \geqslant 3$. For $n = 1, 2$ the graph is not connected.}

\begin{proposition}\label{prop:DGdiag-Sym}
    Let \(\Gamma:=K_{n,n}-n K_2\), where $n\geqslant 3$, and assume that $G^+=\Diag(\Sym(n)\times \Sym(n))$.
    %\(G:=\Aut(\Gamma)=\langle  \Diag(\Sym(n)\times \Sym(n)),\tau\rangle\).
    Then \(D(G)=\lfloor \sqrt{n}\rfloor+1\).
\end{proposition} 
\begin{proof}
 
Let \(k:=D(G)\) and $\Sigma:=\{\Sigma_1 , \Sigma_2, \ldots,\Sigma_k\}$  be a distinguishing  partition of \(V\Gamma\), that is, \(G_{(\Sigma)}=\bigcap_{i=1}^{k}G_{\Sigma_{i}}=1\).
\begin{claim}\label{u_idiffparts}
For any \(1\leqslant i\leqslant  k\), any two vertices of \(\{u_j \mid v_j\in \Sigma_i\}\) lie in two different parts of \(\Sigma\).
\end{claim}
Indeed, otherwise, there would exist \(\Sigma_{i'}\in\Sigma\) and \(u_s, u_t\in \Sigma_{i'}\) with \(v_s,v_t\in \Sigma_i\), and the map \(\pi\coloneqq ((s,t),(s,t))\in G^+\) stabilises each element of  \(\Sigma\).
%Then the transposition \(\tau:=(s,t)\) stabilises \(\Sigma_{1}\) and \(\Sigma_{i}\). 
%Moreover, \(\tau\in \Sigma_{j}\) for \(j\neq 1,i\) since \(\tau\) stabilises \(\{1,\ldots,n\}\setminus\{1,i\}\).
Thus \(\pi\in G_{(\Sigma)}\), which is a contradiction, proving the claim.

By Claim \ref{u_idiffparts}, it follows that \(|\Sigma_{i}\cap \Delta|\leqslant k\) for all \(i\).
Thus \(n=|\Delta|=\sum_{i=1}^{k}|\Sigma_{i}\cap\Delta|\leqslant k^{2}\).
Note that \(k\) is an integer, which implies that \(k\geqslant \lceil\sqrt{n}\,\rceil.\)
%\(k\geqslant \lfloor \sqrt{n}\rfloor+1\) when \(n\) is not a square.

First we deal with the case where \(n\) is not a square. 
Let \(\ell:=\lceil\sqrt{n}\,\rceil=\lfloor \sqrt{n}\rfloor+1\).
Then \(\ell-1\) is the largest integer such that its square is strictly less than \(n\),
thus \((\ell-1)^{2} < n < \ell^2\).
In particular, \(n=(\ell-1)^{2}+p\), where \(1 \leqslant p \leqslant 2\ell-2\).  We show that \(k\leqslant \ell\) by displaying a partition \(\mathcal{P}\) of size \(\ell\) with trivial stabiliser \(G_{(\mathcal{P})}\). This in turn will show that \(k=\ell\) when \(n\) is not a square.
We split the analysis into two cases: (i) \(1 \leqslant p< \ell\); (ii) \(\ell \leqslant p \leqslant 2\ell-2\).

 \medskip\noindent
\emph{Case (i):  \(1 \leqslant p< \ell\)}. \\Let \(\mathcal{P}:=\{P_{1},\ldots, P_{\ell}\}\) where
\begin{equation*} 
P_{i}  \coloneqq
\begin{cases}
   \mathcal{V}_n(i,\ell-1) \cup \mathcal{U}_n(i,\ell-1),& \text{if } 1\leqslant i < \ell;\\ 
   \mathcal{V}_n(i,\ell-1), & \text{if } i=\ell.
\end{cases}
\end{equation*}

This is a partition of size $\ell$ of \(V\Gamma\),
since
\(
\{P_i \cap \Delta' : 1 \leq i \leq \ell\}
= \{\mathcal{U}_n(i, \ell-1) : 1 \leq i \leq \ell - 1\} \cup \{\varnothing\}
\)
is a partition of $\Delta'$ by Lemma \ref{V is partition}(a)
and \(
\{P_i \cap \Delta : 1 \leq i \leq \ell\}
= \{\mathcal{V}_n(i, \ell-1) : 1 \leq i \leq \ell\}
\)
is a partition of $\Delta$
by applying Lemma \ref{V is partition}(b) with $d = \ell$ and $q = \ell - 1$, as $(\ell-1)^2 < n = (\ell-1)^{2} + p \leqslant  \ell \, (\ell - 1).$

% \vspace*{2em}
%\subfile{../tikz/crown_graph/Sn/case1}

%\begin{claim}\label{partstab}
 %Let \(g\in G_{(\mathcal{P})}\). Then $g$ fixes $v_s$ for all $s\leqslant n$.
 %\end{claim}

Note that \(P_{\ell}\subseteq\Delta\) and so $G_{(\mathcal{P})}=G^+_{(\mathcal{P})}$. 
Observe that \(\mathcal{P}\) satisfies the conditions of Lemma \ref{lem:partsaresubsets}. It follows that $G^+_{(\mathcal{P})}$ is trivial.
Thus $\mathcal{P}$ is a distinguishing partition for $G$.
%This implies that \(g\in  G^{+}= \Diag(\Sym(n)\times \Sym(n))\) and so $g=(h,h)$ for some $h\in \Sym(n)$. 
% 
% Suppose $h$ does not fix $s$.  Since \(g\in G_{(\mathcal{P})}\), $v_s^g=v_{s^h}\in P_i$. In particular $0<|s-s^h|<\ell-1$. Then $u_s^g=u_{s^h}$ and $u_s$ are also in the same part $P_j$ and therefore $s\equiv s^h\equiv j \;(\bmod\; \ell - 1)$, contradicting $0<|s-s^h|<\ell-1$. It follows that $h=\mathrm{id}$, and so $g=(\mathrm{id},\mathrm{id}).$
 
%Therefore \(G_{(\mathcal{P})}\) is trivial.

\medskip\noindent
\emph{Case (ii): \(\ell\leqslant p\leqslant 2\ell-2\)}.\\ Let \(\mathcal{P}:=\{P_{1},\ldots, P_{\ell}\}\) where
\begin{equation*}
    P_{i}:=\begin{cases}
       \mathcal{V}_n(i,\ell)\cup \mathcal{U}_n(i,\ell)&\text{if }1\leqslant i\leqslant \ell-2\\
        
        \mathcal{V}_n(i,\ell)\cup  \mathcal{U}_n(\ell,\ell)&\text{if } i = \ell-1\\
        
        \mathcal{V}_n(i,\ell)\cup  \mathcal{U}_n(\ell-1,\ell)&\text{if }i=\ell.
    \end{cases}
\end{equation*}

Arguing similarly to the previous case, we see that this is a partition of  \(V\Gamma\) of size $\ell$,
by applying Lemma \ref{V is partition}(b) with $d = \ell$ and $q = \ell$, as $(\ell-1)\ell<n=(\ell-1)^{2}+p\leqslant \ell^2.$

Note that \(|P_{\ell - 1}\cap\Delta'|=\ell-1<\ell=|P_{\ell - 1}\cap\Delta|\),
and so $G_{(\mathcal{P})}=G^+_{(\mathcal{P})}$. 
Observe that \(\mathcal{P}\) satisfies the conditions of Lemma \ref{lem:partsaresubsets}. It follows that $G^+_{(\mathcal{P})}$ is trivial.
Thus $\mathcal{P}$ is a distinguishing partition for $G$.

% \vspace*{2em}
%\subfile{../tikz/crown_graph/Sn/case2}
%We now show that $G_{(\mathcal{P})}$ is trivial. 
 %Let \(g\in G_{(\mathcal{P})}\).
%Observe that \(|P_{\ell - 1}\cap\Delta'|=\ell-1<\ell=|P_{\ell - 1}\cap\Delta|\). Since $g$ preserves $P_{\ell - 1}$, it follows that \(g=(h,h)\in G^{+}\).
%
%Suppose $h$ does not fix $s$.  Since \(g\in G_{(\mathcal{P})}\), we get that $0<|s-s^h|<\ell$. Then $u_s^g=u_{s^h}$ and $u_s$ are also in the same part $P_j$ and therefore $s\equiv s^h \;(\bmod\; \ell )$, contradicting $0<|s-s^h|<\ell$. It follows that $h=\mathrm{id}$, and so $g=(\mathrm{id},\mathrm{id}).$ 
%Therefore \(G_{(\mathcal{P})}\) is trivial.

%With a similar argument used in Case (i) we conclude that \(G_{(\mathcal{P})}\) is trivial.

From these partitions of size \(\ell\)  we see that \(D(G)=\lfloor \sqrt{n}\rfloor+1\) when \(n\) is not a square.
Thus it remains to show that the result also holds for \(n=\ell^{2}\) a square.
Note that \(k^{2}\geqslant n\), and so \(k\geqslant \ell\). 

\medskip
First assume that \(k=\ell\).
Recall that \(|\Sigma_{i}\cap\Delta|\leqslant k=\ell\)
and \(\sum_{i=1}^{\ell}|\Sigma_{i}\cap\Delta|=n=\ell^{2}\).
Hence \(|\Sigma_{i}\cap \Delta|=\ell\) for \(1\leqslant i\leq\ell\). 
Similarly one can show that \(|\Sigma_{i}\cap\Delta'|=\ell\) for \(1\leqslant i\leqslant \ell\).
By Claim \ref{u_idiffparts}, for any \(i\), any two vertices of \(\{u_j|v_j\in \Sigma_i\}\) lie in two different parts of \(\Sigma\), so \(\{u_j \mid v_j\in \Sigma_i\}\) has exactly one vertex from each element of \(\Sigma\).
%
%
%Thus, up to reordering of the vertices, we may assume that \(\Sigma_{i}=\mathcal{V}_n(i,\ell)\cup \mathcal{U}_n(i,\ell)\) for \(1\leqslant i\leqslant \ell\) where \(|\mathcal{V}_n(i,\ell)|=|\mathcal{U}_n(i,\ell)|=\ell\) for each \(i\).
%{\color{purple} Not sure if we want to explain the argument here more rigorously or not. I had to think a few minutes why we could reorder the vertices like that by using a permutation $(x,x) \in \Diag(\Sym(n) \times \Sym(n))$. }
%{\color{red} I am trying something else. Is this clearer or more confusing than reordering?}
Therefore the map $\alpha:\{1,2,\ldots,\ell^2\}\mapsto \{1,2,\ldots,\ell\}^2$ defined by $\alpha(i)=(x,y)$ if $v_i\in \Sigma_x$ and $u_i\in \Sigma_y$, is a bijection.
Let $h$ be the element of $\Sym(n)$, defined by $\alpha(i^h)=(y,x)$ where $\alpha(i)=(x,y)$. Clearly $h$ is an involution.
 Let \(g:=(h,h)\tau\in G\), so that  $v_i^g=(v_{i^h})^\tau=u_{i^h}$ and $u_i^g=(u_{i^h})^\tau=v_{i^h}$. It follows that $g$ is an involution.
 %%%%%{\color{red} here again we need that $\tau \in G$, after I changed the hypothesis slightly. Really need in the preliminary something about $G\setminus G^+$ when $G^+$ is diagonal, see text in blue.}

 We claim that $g\in G_{({\Sigma})}.$
Let $v_i\in\Sigma_x\cap \Delta$, thus $\alpha(i)=(x,y)$ for some $y$, and so $\alpha(i^h)=(y,x)$, and so $v_i^g=u_{i^h}\in \Sigma_x\cap\Delta'$.
Therefore \((\Sigma_{x}\cap\Delta)^{g}=\Sigma_{x}\cap\Delta'\).
Since \(g \) is an involution we also have that \((\Sigma_{x}\cap\Delta')^{g}=\Sigma_{x}\cap\Delta\).
This shows $\Sigma$ is not a distinguishing partition, a contradiction.  Thus
\(k\geqslant \ell+1\).

%Let \(h\in H\cong \Sym(n)\) mapping \(v_{(i-1)\ell+j}\) to \(v_{(j-1)\ell+i}\) for each \(1\leqslant i,j\leqslant \ell\).  Let \(g:=(h,h)\tau\in G\). Note that \(|h|=2\) and so \(|g|=2\).

%Then  \begin{align*}           
%    (\Sigma_{i}\cap\Delta)^{g}&=\mathcal{V}_n(i,\ell)^g=\{v_{(i-1)\ell+j} \mid  j=1,\ldots, \ell\}^g\\
%&=  \{u_{(i-1)\ell+j} | j=1,\ldots, \ell\}^h  \\
%&=  \{u_{(j-1)\ell+i} \mid j=1,\ldots, \ell\}\\
%&=\mathcal{U}_n(i,\ell)=\Sigma_{i}\cap\Delta'
%\end{align*}
%Since \(g \) is an involution we also have that \((\Sigma_{i}\cap\Delta')^{g}=\Sigma_{i}\cap\Delta\).

%Thus \(g\) stabilises \(\Sigma_{1},\ldots, \Sigma_{\ell}\) and \(g\in G_{(\mathcal{P})}\). 
%This contradicts the  assumption   that $D(G)=k$. Hence \(k\geq\ell+1\).

We will now construct a distinguishing partition with \(\ell+1\) parts, which shows that \(D(G)=\ell+1=\sqrt{n}+1=\lfloor \sqrt{n}\rfloor +1\) when \(n\) is a square.
Consider \(\mathcal{P}:=\{P_{1},\ldots, P_{\ell+1}\}\)  whereby 
\begin{equation}
P_{i}:=
\begin{cases}
    \mathcal{V}_n(i,\ell)\cup\mathcal{U}_n(i,\ell) \setminus \{u_{\ell^{2}}\} & \text{if } 1\leqslant i\leqslant \ell\\
    \{u_{\ell^{2}}\} & \text{if } i=\ell+1.
\end{cases}
\end{equation}
This is clearly a partition of size $\ell+1$ of \(V\Gamma\),
by applying Lemma \ref{V is partition}.

%\vspace*{2em}
%\subfile{../tikz/crown_graph/Sn/case3}

Note that  \(P_{\ell+1}\subseteq \Delta'\),
and so $G_{(\mathcal{P})}=G^+_{(\mathcal{P})}$. 
Observe that \(\mathcal{P}\) satisfies the conditions of Lemma \ref{lem:partsaresubsets}. It follows that $G^+_{(\mathcal{P})}$ is trivial.
Thus $\mathcal{P}$ is a distinguishing partition for $G$.
%Then since \(P_{\ell+1}\subseteq \Delta'\), it follows that \(G_{(\mathcal{P})}\leqslant G_{\Delta'}= G^{+}\).
%Applying a similar argument as in Case (i) we conclude that \(G_{(\mathcal{P})}\)  is trivial. 
%Hence \(D(G)=\ell+1=\lfloor \sqrt{n}\rfloor+1\), proving the result for the case \(n\) square.
\end{proof}

\begin{proposition}\label{prop:DGdiag-Alt}
Let \(\Gamma:=K_{n,n}-n K_2\), where $n\geqslant3$, and assume $G^+=\Diag(\Alt(n)\times \Alt(n)).$
%\(G:=\langle  \Diag(\Alt(n)\times \Alt(n)),\tau\rangle<\Aut(\Gamma)\). 
Then \(D(G)=\lceil \sqrt{n - 1}\rceil\).
\end{proposition} 
%{\color{red} change group to say $G$ transitive and $G^+=\Diag(\Alt(n)\times \Alt(n))$? This covers also another group, namely \(G:=\langle  \Diag(\Alt(n)\times \Alt(n)),((1,2),(1,2))\tau\rangle\)

\begin{proof}

Let \(k:=D(G)\) and \(\Sigma=\{\Sigma_{1},\ldots, \Sigma_{k}\}\) be a distinguishing partition of \(V\Gamma\), so that \(G_{(\Sigma)}=\bigcap_{i=1}^{k}G_{\Sigma_{i}}=1\).

\begin{claim}\label{u_iatmosttwo}
For any \(1\leqslant i\leqslant  k\), either
\begin{enumerate}
\item[(a)]  any two vertices of \(\{u_j\mid v_j\in \Sigma_i\}\) lie in two different parts of \(\Sigma\), or 
\item[(b)]  there exist $j_1\neq j_2$ such that $v_{j_1},v_{j_2}\in \Sigma_i$, $u_{j_1}$, $u_{j_2}$ are in the same part, and any two vertices of \(\{u_j \mid v_j\in \Sigma_i, j\neq j_2\}\) lie in two different parts of \(\Sigma\). 
\end{enumerate}
Moreover, at most one $i$ satisfies (b). The same result holds if we swap the $u$'s and the $v$'s.
\end{claim}
Let \(1\leqslant i\leqslant  k\) and assume (a) does not hold. 
Thus there are vertices of \(\{u_j|v_j\in \Sigma_i\}\) that are in
the same part $\Sigma_{i'}$ of \(\Sigma\), say $u_{j_1}$, $u_{j_2}$.
If $|\Sigma_i\cap\Delta|=2$, then (b) trivially holds. Assume $|\Sigma_i\cap\Delta|>2$.
Consider $u_j$ in \(\{u_j|v_j\in \Sigma_i\}\) with $j\notin\{j_1,j_2\}$. 
If $u_j\in \Sigma_{i'}$ then the map $\left((j,j_1,j_2),(j,j_1,j_2)\right)\in G^+$ stabilises each element of \(\Sigma\), a contradiction. So $u_j\notin \Sigma_{i'}$.
If $u_j$ is in the same part as $u_\ell$ where $v_\ell\in \Sigma_i$, then  $\left((j,\ell)(j_1,j_2),(j,\ell)(j_1,j_2)\right)\in G^+$ stabilises each element of \(\Sigma\), again a contradiction. Thus (b) holds.

Now assume (b) holds for both $i$ and $i'$. Thus there exist $j_1 \neq j_2$ and $m_1 \neq m_2$ such that $v_{j_1},v_{j_2}\in \Sigma_i$, $u_{j_1}$, $u_{j_2}$ are in the same part, $v_{m_1},v_{m_2}\in \Sigma_{i'}$, $u_{m_1}$, $u_{m_2}$ are in the same part.
Then $\left(({j_1},{j_2})({m_1},{m_2}),({j_1},{j_2})({m_1},{m_2})\right)\in G^+$ stabilises each element of \(\Sigma\), again a contradiction. This proves the claim.

%Hence the Claim is proven. 
It follows from Claim \ref{u_iatmosttwo} that there exists at most one $i \in \{1, \dots, k\}$ such that $|\Sigma_i \cap \Delta| >k$, and if there exists such an $i$, then $|\Sigma_i \cap \Delta| = k+1$. 
In particular
\[
n \leqslant k (k - 1) + (k + 1) = k^2 + 1,
\]
and thus $k \geqslant \lceil \sqrt{n - 1} \rceil$.

Now let $\ell \coloneqq \lceil \sqrt{n - 1} \rceil$. Note $\ell\geq 2$ since $n\geq 3$.
Then \(\ell-1\) is the largest integer such that \((\ell-1)^{2}<n-1\).
In particular, \(n-1=(\ell-1)^{2}+p\), where \(1 \leqslant p \leqslant 2\ell-1\).  We will show that \(k\leqslant \ell\) by displaying a partition \(\mathcal{P}\) of size \(\ell\) with trivial stabiliser \(G_{(\mathcal{P})}\). This in turn will show that \(k=\ell\).

We split the analysis into two cases: (i) \(1 \leqslant p\leqslant \ell-2\); (ii) \(\ell-1 \leqslant p \leqslant 2\ell-1\).

\emph{Case (i): \(1 \leqslant p\leqslant \ell-2\)}. \\
Consider \(\mathcal{P}:=\{P_{1},\ldots, P_{\ell}\}\) where
\begin{equation*} 
P_{i}  \coloneqq
\begin{cases}
   \mathcal{V}_n(i,\ell-1) \cup \mathcal{U}_n(i,\ell-1)& \text{if } 1\leqslant i \leqslant \ell-2\\ 
    \mathcal{V}_n(\ell-1,\ell-1) \cup \{v_{(\ell-1)^2+1}\}\cup \mathcal{U}_n(\ell-1,\ell-1)& \text{if } i= \ell-1\\ 
   \mathcal{V}_n(\ell,\ell-1)\setminus \{v_{(\ell-1)^2+1}\} & \text{if } i=\ell.
\end{cases}
\end{equation*} 

% \vspace*{2em}
%\subfile{../tikz/crown_graph/An/case1}
Note that $|\mathcal{V}_n(i,\ell-1)|=\ell-1$ for all $i<\ell$, $|\mathcal{V}_n(\ell,\ell-1)|=1+p\leqslant \ell-1$ and $v_{(\ell-1)^2+1}\in \mathcal{V}_n(\ell,\ell-1)$.
Thus $|P_{i}\cap\Delta|=\ell-1$ for $1\leqslant i \leqslant \ell-2$, $|P_{\ell-1}\cap\Delta|=\ell$, $|P_{\ell}\cap\Delta|=p$. Using Lemma~\ref{V is partition}(a), it easily follows that \(\mathcal{P}\) is indeed a partition of size $\ell$. 

Let \(g\in G_{(\mathcal{P})}\). 
Observe that \(P_{\ell}\subset\Delta\). This implies that \(G_{(\mathcal{P})}\leqslant G_{\Delta}= G^{+}\), and so $g=(h,h)$ for some $h\in \Alt(n)$. 
  Let $1\leqslant s\leqslant n$. Note $v_s^g=v_{s^h}$ and $v_s$ are in the same part $P_i$. In particular $ |s-s^h|<\ell-1$ if $i\neq \ell-1$ and $ |s-s^h|<\ell$ if $i=\ell-1$. On the other hand, $u_s^g=u_{s^h}$ and $u_s$ are also in the same part. Therefore $s\equiv s^h \;\pmod {\ell-1}$. It follows that $s^h=s$, unless $\{s,s^h\}=\{(\ell-2)(\ell-1)+1,(\ell-1)^2+1\}$. Since $h$ is an even permutation, it follows that $h=\mathrm{id}$, and  $g=(\mathrm{id},\mathrm{id}).$
  Therefore $G_{(\mathcal{P})}$ is trivial.   

 %\begin{claim}\label{claim5}
 %Let \(g\in G_{(\mathcal{P})}\). Then $g$ fixes $v_s$ for all $s\leqslant n$.
 %\end{claim}
 % Suppose $v_s\in P_i$ is not fixed by $g$.  Since \(g\in G_{(\mathcal{P})}\), $v_s^g\in P_i$, and $v_s^g=v_t$ for some $t$, since $g\in G^+$. In particular, if $i\neq \ell-1$, then $0<|s-t|<\ell-1$. Thus $u_s^g=u_t$ and $u_s$ are also in the same part $P_j$ and therefore $s\equiv t\equiv j \mod(\ell-1)$, contradicting $0<|s-t|<\ell-1$. Hence $i=\ell-1$.  Since $u_t$ and $u_s$ are in the same part $P_j$, we must have $s\equiv t\equiv j \pmod{\ell-1}$ which implies that $\{s,t\}=\{(\ell-2)(\ell-1)+1,(\ell-1)^2+1\}$, and $g$ fixes all other vertices in $\Delta$.
 %Therefore $g^\Delta$ is a transposition, contradicting the fact that  $G^+= \Diag(\Alt(n)\times \Alt(n))$. 
 %This proves the claim.

\emph{Case (ii): \(\ell-1 \leqslant p \leqslant 2\ell-1\)}.\\
Consider \(\mathcal{P}:=\{P_{1},\ldots, P_{\ell}\}\) where
\begin{equation*} 
P_{i}  \coloneqq
\begin{cases}
   \mathcal{V}_n(i,\ell) \cup \mathcal{U}_n(i,\ell)& \text{if } 1\leqslant i \leqslant \ell-2\\ 
    \mathcal{V}_n(\ell-1,\ell) \cup \{v_n\}\cup \mathcal{U}_n(\ell,\ell)& \text{if } i= \ell-1\\ 
   \mathcal{V}_n(\ell,\ell)\setminus \{v_n\} \cup \mathcal{U}_n(\ell-1,\ell)& \text{if } i=\ell.
\end{cases}
\end{equation*} 

Note that $|\mathcal{V}_n(i,\ell)|=\ell$ for all $i\leqslant \ell-1$ since $p>\ell-2$. When $p<2\ell-1$, $|\mathcal{V}_n(\ell,\ell)|=p-(\ell-2)\leqslant \ell$ and $v_{n}\in \mathcal{V}_n(\ell,\ell)$. When $p=2\ell-1$, $n=\ell^2+1$ so that $|\mathcal{V}_n(\ell,\ell)|=\ell$ and $v_{n}\notin \mathcal{V}_n(i,\ell)$ for any $1\leq i\leq \ell$. 
Thus $|P_{i}\cap\Delta|=\ell$ for $1\leqslant i \leqslant \ell-2$, $|P_{\ell-1}\cap\Delta|=\ell+1$.
 When $p<2\ell-1$, $0\leqslant |P_{\ell}\cap\Delta|=p+1-\ell\leqslant\ell-1$.
 When $p=2\ell-1$, $|P_{\ell}\cap\Delta|=\ell$.
 Using Lemma~\ref{V is partition}(a), it easily follows that \(\mathcal{P}\) is indeed a partition of size $\ell$. 

Let \(g\in G_{(\mathcal{P})}\). 
Observe that $|P_{\ell-1}\cap\Delta'|=|\mathcal{U}_n(\ell,\ell)|$. Since $n\leqslant \ell^2+1$ and $\ell\geqslant 2$, we see that   $|P_{\ell-1}\cap\Delta'|\leqslant \ell$ so $|P_{\ell-1}\cap\Delta|\neq |P_{\ell-1}\cap\Delta'|$. This implies that \(G_{(\mathcal{P})}\leqslant G_{P_{\ell-1}}\leqslant G^{+}\), and so $g=(h,h)$ for some $h\in \Alt(n)$. 
  Let $1\leqslant s\leqslant n$. Note $v_s^g=v_{s^h}$ and $v_s$ are in the same part $P_i$. In particular $ |s-s^h|<\ell$ if $i\neq \ell-1$. On the other hand, $u_s^g=u_{s^h}$ and $u_s$ are also in the same part. Therefore $s\equiv s^h \;\pmod {\ell}$. It follows that $s^h=s$, unless $i=\ell-1$ and $\{s,s^h\}=\{n,t\}$ where $t$ is the unique element in  $\mathcal{V}_n(\ell-1,\ell)$ such that  $t\equiv n \;\pmod {\ell}$ . Since $h$ is an even permutation, it follows that $h=\mathrm{id}$, and  $g=(\mathrm{id},\mathrm{id}).$
  Therefore $G_{(\mathcal{P})}$ is trivial.  
% 
 %Note on the other hand $|P_{\ell-1}\cap\Delta'|=|\mathcal{U}_n(\ell,\ell)|$. Since $n\leqslant \ell^2+1$ and $\ell\geqslant 2$, we see that   $|P_{\ell-1}\cap\Delta'|\leqslant \ell$ so $|P_{\ell-1}\cap\Delta|\neq |P_{\ell-1}\cap\Delta'|$. This implies that \(G_{P}\leqslant G_{P_{\ell-1}}\leqslant G_{\Delta}= G^{+}\). 
% 
% 
% 
%\begin{claim}
%Let \(g\in G_{(\mathcal{P})}\).  Then $g$ fixes $v_s$ for all $s\leqslant n$.
%\end{claim}
%Suppose $v_s\in P_i$ is not fixed by $g$.  Since \(g\in G_{P}\), $v_s^g\in P_i$, and $v_s^g=v_t$ for some $t$, since $g\in G^+$. In particular, if $i\neq \ell-1$, then $0<|s-t|<\ell$. Thus $u_s^g=u_t$ and $u_s$ are also in the same part $P_j$ and therefore $s\equiv t\equiv j \mod\ell$, contradicting $0<|s-t|<\ell$. Hence $i=\ell-1$.  Since $u_t$ and $u_s$ are in the same part $P_j$, we must have $s\equiv t\equiv j \mod\ell$ which implies that $s=n$ or $t=n$, and $g$ fixes all other vertices in $\Delta$.
%Therefore $g^\Delta$ is a transposition,  contradicting the fact that  $G^+= \Diag(\Alt(n)\times \Alt(n))$. 
%This proves the claim.
% 
%We conclude again that \(G_{(\mathcal{P})}\)  fixes \(\Delta\) pointwise, and so \(G_{(\mathcal{P})}=G^+_{(\mathcal{P})}=1\) since $G^+$ is diagonal.
%In every case,  \(G_{P}=G_{P}^{+}\leqslant\bigcap P_{i=1}^{n}G^{+}_{v_{i}}=1\). Thus \(G_{P}=1\). 
Hence \(D(G)=\ell\), proving the result.
\end{proof}

\section{Proofs of the main theorems}

Theorem~\ref{intro thm: dno of crown} follows directly from Theorem~\ref{intro thm: table values}. Theorem~\ref{intro thm: table values} itself follows from 
Proposition \ref{prop:notfaithful}, 
Proposition \ref{prop:DGdiag-outer}, 
Proposition \ref{prop:n=4}, 
Proposition \ref{prop:DGdiag-Sym}, and 
Proposition \ref{prop:DGdiag-Alt}. 
Finally, Theorem~\ref{main} follows  from 
Proposition \ref{group-cases}.

%\nocite{}
%\begin{thebibliography}{}
%\bibitem{CDMR} L.~Chen, A.~Devillers, L.~Morgan and F.~Rober, The distinguishing number of {$2$-arc-transitive} bipartite graphs, \emph{in preparation}.

%\bibitem{CNS} P.J.~Cameron, P.N.~Neumann and J.~Saxl, On groups with no regular orbits on the set of subsets, \emph{Archiv der Mathematik} 43 (1984), 295--296.

%\bibitem{CT2006} K.L.~Collins and A.N.~Trenk, The distinguishing chromatic numer, \emph{Electronic Journal of Combinatorics} 13 (2006), \#R18.

%\bibitem{DMH2018} A.~Devillers, L.~Morgan and S. Harper, The distinguishing numer of quasiprimitive and semiprimitive groups, \emph{Archiv der Mathematik} 113.2 (2019), 127--139.

%\bibitem {P2009} C.~Laflamme and K.~Seyffarth, Distinguishing chromatic numbers of bipartite graphs, \emph{Electronic Journal of Combinatoris} (2009).

%\bibitem{PS} C.E.~Praeger and C.~Schneider, \emph{Permutation groups and Cartesion decompositions}, Cambridge University Press, Cambridge, 2018.

%\bibitem{Ty}
%J.~S. Tymoczko, Distinguishing numbers for graphs and groups, Electron. J. Combin. {\bf 11} (2004), no.~1, Research Paper 63.

%\end{thebibliography}
 \printbibliography[heading=bibintoc,title={Bibliography}]

@article{CT2006,
author={Collins, K. L.  and Trenk, A. N.},
title={The distinguishing chromatic number},
journal={Electronic Journal of Combinatorics},
number={13},
year={2006},
issue={{$\#$}R16}
}

@article{P2009,
author = {Laflamme, C and Seyffarth, K},
year = {2009},
month = {},
pages = {},
title = {Distinguishing chromatic numbers of bipartite graphs},
volume = {},
journal = {Electronic Journal of Combinatorics},
doi = {}
}

@article{DMH2018,
title = {The distinguishing number of quasiprimitive and semiprimitive groups},
author = {Devillers, Alice and Harper, Scott and Morgan, Luke},
year = {2019},
month = {08},
day = {1},
doi = {10.1007/s00013-019-01324-7},
volume = {113},
pages = {127--139},
journal = {Archiv der Mathematik},
issn = {0003-889X},
publisher = {Birkhauser Verlag Basel},
number = {2}
}

@article{CNS1994,
author = {Cameron, Peter J. and Neumann, Peter M. and Saxl, Jan},
date = {1984/10/01},
doi = {10.1007/BF01196649},
id = {Cameron1984},
isbn = {1420-8938},
journal = {Archiv der Mathematik},
number = {4},
pages = {295--296},
title = {On groups with no regular orbits on the set of subsets},
url = {https://doi.org/10.1007/BF01196649},
volume = {43},
year = {1984}
}

@article {ConderTucker,
    AUTHOR = {Conder, Marston and Tucker, Thomas},
     TITLE = {Motion and distinguishing number two},
   JOURNAL = {Ars Math. Contemp.},
  FJOURNAL = {Ars Mathematica Contemporanea},
    VOLUME = {4},
      YEAR = {2011},
    NUMBER = {1},
     PAGES = {63--72},
      ISSN = {1855-3966,1855-3974},
   MRCLASS = {05E18 (05E15 20B05)},
  MRNUMBER = {2781021},
MRREVIEWER = {Robert\ Jajcay},
       DOI = {10.26493/1855-3974.192.531},
       URL = {https://doi.org/10.26493/1855-3974.192.531},
}

@book {PS,
    AUTHOR = {Praeger, Cheryl E. and Schneider, Csaba},
     TITLE = {Permutation groups and {C}artesian decompositions},
    SERIES = {London Mathematical Society Lecture Note Series},
    VOLUME = {449},
 PUBLISHER = {Cambridge University Press, Cambridge},
      YEAR = {2018},
     PAGES = {xiii+323},
      ISBN = {978-0-521-67506-2},
   MRCLASS = {20B15 (05A05 05C25 20B05 20B07)},
  MRNUMBER = {3791829},
MRREVIEWER = {Pablo\ Spiga},
       DOI = {10.1017/9781139194006},
       URL = {https://doi.org/10.1017/9781139194006},
}

@article {klavzar,
    AUTHOR = {Klav\v zar, Sandi and Wong, Tsai-Lien and Zhu, Xuding},
     TITLE = {Distinguishing labellings of group action on vector spaces and
              graphs},
   JOURNAL = {Journal of Algebra},
  FJOURNAL = {Journal of Algebra},
    VOLUME = {303},
      YEAR = {2006},
    NUMBER = {2},
     PAGES = {626--641},
      ISSN = {0021-8693,1090-266X},
   MRCLASS = {05C25 (20B25)},
  MRNUMBER = {2255126},
MRREVIEWER = {Anthony\ Weaver},
       DOI = {10.1016/j.jalgebra.2006.01.045},
       URL = {https://doi.org/10.1016/j.jalgebra.2006.01.045},
}

@article{CDMR,
author={Lei Chen and Alice Devillers and Luke Morgan and Friedrich Rober},
title={The distinguishing number of bipartite $2$-arc-transitive graphs},
series={In Preparation},
%year={}
}

@article{P2006,
author = {Chan, Melody},
title = "{The distinguishing number of the direct product and wreath product action}",
journal = {Journal of Algebraic Combinatorics},
volume = {},
number = {},
pages = {331-345},
year = {2006},
month = {},
issn = {},
doi = {0.1007/s10801-006-0006-7},

}

@article{Ty,
author = {Tymoczko, Julianna},
title = {Distinguishing numbers for graphs and groups},
journal = { Electronic Journal of Combinatorics},
volume = {11},
year = {2004},
DOI = { https://doi.org/10.37236/1816},

}

@article {RussellSundaram,
    AUTHOR = {Russell, Alexander and Sundaram, Ravi},
     TITLE = {A note on the asymptotics and computational complexity of
              graph distinguishability},
   JOURNAL = { Electronic Journal of Combinatorics},
  FJOURNAL = {Electronic Journal of Combinatorics},
    VOLUME = {5},
      YEAR = {1998},
     PAGES = {Research Paper 23, 7},
      ISSN = {1077-8926},
   MRCLASS = {05C85 (68Q15)},
  MRNUMBER = {1617449},
       DOI = {10.37236/1361},
       URL = {https://doi.org/10.37236/1361},
}
\end{document}